\documentclass[reqno]{amsart}

\newtheorem{theorem}{Theorem}[section]
\newtheorem{lemma}[theorem]{Lemma}

\theoremstyle{definition}

\usepackage[T1]{fontenc}
\usepackage[utf8]{inputenc}
\usepackage{lmodern}
\usepackage{textcomp}
\usepackage{microtype}
\usepackage{tikz}
\usepackage{mathtools, bm, nccmath}
\usepackage{comment}

\usepackage[english]{babel}
\usepackage{csquotes}
\usepackage[export]{adjustbox}
\usepackage{circuitikz}

\usepackage{tabularx}

\usepackage{graphicx}
\usepackage{setspace}

\usepackage{amsmath}
\usepackage{amsfonts}
\usepackage{amssymb}
\usepackage{amsthm}
\usepackage{xcolor}%
\usepackage{soul}
\usepackage{graphicx} 
\usepackage{lipsum}
\numberwithin{equation}{section}

\newcommand*{\N}{\mathbb{N}}

\newcommand*{\R}{\mathbb{R}}
\newcommand*{\C}{\mathbb{C}}

\newcommand*{\E}{\mathbb{E}}

\newcommand {\e} {\varepsilon}

\newcommand {\ve} {\varepsilon}

\makeatletter\def\blfootnote{\xdef\@thefnmark{}\@footnotetext}\makeatother

\newcommand{\probP}{\text{I\kern-0.15em P}}

\usepackage{thmtools}

\declaretheorem[
	name=Theorem,
	numberwithin=section
	]{thm}

\declaretheorem[
	name=Remark,
	style=remark,
	numbered=no
	]{rem}

\usepackage[pdftex,pdfpagelabels]{hyperref}

\numberwithin{equation}{section}

\makeatletter
\@namedef{subjclassname@2020}{\textup{2020} Mathematics Subject Classification}
\makeatother

\calclayout

\begin{document}
\title[LACUNARY SERIES, NONLINEAR FUNCTIONALS AND BANACH SPACE STRUCTURE]
{\bf LACUNARY SERIES, NONLINEAR FUNCTIONALS AND BANACH SPACE STRUCTURE}

\author[I. Berkes, E. Stefanescu, and R. Tichy ]{Istvan Berkes, Eduard Stefanescu, and Robert Tichy}
\address{Institut f\"ur Analysis und Zahlentheorie, TU Graz, Steyrergasse 30, 8010 Graz, Austria}
\email{\href{mailto:berkes.istvan@renyi.hu }{berkes@renyi.hu }}
\email{\href{mailto:eduard.stefanescu@tugraz.at}{eduard.stefanescu@tugraz.at}}
\email{\href{mailto:tichy@tugraz.at}{tichy@tugraz.at}}


\begin{abstract}
In a previous paper \cite{BT} we studied the asymptotic behavior of $\| \sum_{k=1}^N a_k X_{n_k}\|_p$ for lacunary sequences $(X_{n_k})$  of random variables in $L_p$ and used the result to give a necessary and sufficient condition for the first alternative in the Kadec-Pe{\l}czynski theorem in the case $1\le p<2$. In the present paper we extend this result for nonlinear functionals $f_k (a_1 X_{n_1}, \ldots, a_k X_{n_k})$, establishing a uniform version of the  subsequence principle of Aldous \cite{ald}. Moreover, we prove Kadec-Pe{\l}czynski type theorems in Orlicz spaces $L_\psi$.

\end{abstract}

\maketitle

\section{Introduction}

Call two sequences $(x_n)$ and $(y_n)$ defined in the Banach spaces $(B_1, \| \cdot \|_{B_1})$, resp.\ $(B_2, \| \cdot \|_{B_2})$
{\it equi\-valent} if for some constant $K>0$ we have
$$
K^{-1}\big\Vert \sum_{i=1}^n a_i x_i\big\Vert_{B_1} \le \big\Vert
\sum_{i=1}^n a_i y_i \big\Vert_{B_2} \le K\big\Vert \sum_{i=1}^n a_i
x_i\big\Vert_{B_1}
$$
for every $n\ge 1$ and every $(a_1,\ldots,a_n)\in \R^n$.
By a classical theorem of Kadec and
Pe{\l}czynski \cite{KP}, for any normalized weakly null sequence
$(x_n)$ in $L_p(0, 1)$, $p>2$
there exists a subsequence $(x_{n_k})$ which is
equivalent to the unit vector basis of $\ell_2$ or $\ell_p$. Thus every infinite
dimensional closed subspace of $L_p(0,1)$, $p>2$ contains a
subspace isomorphic to $\ell_2$ or $\ell_p$.
In the case when $\{|x_n|^p, \, n\ge 1\}$ is uniformly integrable,
the first alternative holds, while if the functions $(x_n)$ have disjoint
support, the second alternative holds trivially. The general case follows via a subsequence splitting argument
as in \cite{KP}.

The proofs in \cite{KP} depend on the study of subspaces of $L_p$ spanned by independent or lacunary sequences of random variables which, in turn, requires sharp inequalities for the norm of partial sums of such sequences. Classical examples are the Khinchin, Marcinkiewicz-Zygmund, Rosenthal inequalities and their variants, see e.g. \cite{KH,MZ2,MZ1,R1,R2}. In recent decades, the theory has been extended to rearrangement invariant, in particular Orlicz spaces, see e.g. \cite{A2} and the references therein. In particular, there is a characterization of rearrangement invariant spaces in which the Khinchin and Marcinkiewicz inequalities hold, see \cite{A1,A2,RS}.  A simple and direct proof in the special cases of Lorentz and Orlicz spaces was given by the authors in a preprint; see \cite{BST}. For a general overview and further results, see \cite{ALM,KLM,Mal} and the references therein.

For $1\le p<2$ the situation becomes more complex and many basic problems about the subspace structure of $L_p$ remain open until today.
But in \cite{BT} a necessary and sufficient condition was given for the first alternative in the Kadec-Pe{\l}czyinski theorem for the case $1\le p<2$, which we formulate below.  Using the terminology of \cite{bero}, call a sequence of $(X_n)$ of random variables {\it determining} if it has a limit distribution relative to any set $A$ in the probability space with $P (A) > 0$, i.e.,
there exists a distribution function $F_A$ such that $\lim P (X_n \le t | A) = F_A(t)$ as $n\to\infty$ for all continuity points $t$ of $F_A$. Here $P (\cdot|A)$ denotes conditional probability given $A$. As is shown in \cite{ald,bero}, there exists a random measure $\mu$ (i.e.\ a measurable map from the underlying probability space $(\Omega, {\mathcal{F}}, P )$ to $(\mathcal{M}, \pi)$, where ${\mathcal{M}}$ is the set of probability measures on $\mathbb{R}$ and $\pi$  is the Prohorov distance), such that for any $A$ with  $P(A)>0$ and any continuity point $t$ of $F_A$ we have $F_A(t) = \E_A (\mu (-\infty, t]),$ where $\E_A$ denotes conditional expectation given $A$. We call $\mu$  the {\it limit random measure} of $(X_n)$. Let $(Y_n)$ be a sequence of random variables on $(\Omega, {\mathcal{F}}, P )$  such that, conditionally on  $(X_1, X_2, \ldots, \mu)$, the $Y_n$ are i.i.d.\ with conditional distribution $\mu$. See \cite{ald}, p.\ 72; note that the construction of $(Y_n)$  requires an enlargement of the probability space. Clearly, $(Y_n)$ is an exchangeable sequence, we call it the {\it limit exchangeable sequence} of $(X_n)$.

A Young function, or Orlicz function, is a convex mapping
$\psi:[0,\infty)\to[0,\infty)$
satisfying
$\psi(x)/x\to\infty$ as $x\to\infty,$
 and
$\psi(x)/x\to 0$ as $x\to 0.$
In particular, every Young function is non-decreasing, continuous and locally Lipschitz on \((0,\infty)\), and hence differentiable almost everywhere. Moreover, its left and right derivatives, denoted by \(\psi'_-\) and \(\psi'_+\), exist at every point of \((0,\infty)\), see e.g. \cite[ch. 1]{KPP}.

Let $(\Omega, {\mathcal{F}}, P )$ be a probability space. 
The Orlicz norm of a measurable function $f:\Omega\to\R$ or $\C$ is given by
\begin{equation}\label{olnorm}
    \| f \|_{L^{\psi}(Y)} = \inf \left\{ \lambda > 0 \ : \ \int_{\Omega} \psi\left(\frac{|f(x)|}{\lambda}\right) dP(x) \leq 1 \right\}.
\end{equation}
The \emph{Orlicz-space} $L^\psi$ with assigned Young-function $\psi$ consists of the set of measurable functions $f$ such that $\|f\|_{L^\psi}$ is finite.
For $1\le p<\infty$, if $\psi(x)=x^p$, then $L^\psi=L^p$.
Since we only consider finite-measure spaces,  we have $\psi\ll\varphi$ if and only if $ L^{\varphi}\subset L^{\psi}$ and then
\begin{equation}\label{embedding}
    \|\cdot\|_{L^{\psi}}\ll\|\cdot\|_{L^{\varphi}},
\end{equation}
see e.g. Chapter $2$ of \cite{SFMA}. The notation \(f \ll g\) means that there exists a constant \(C>0\) such that \(|f| \le C g\) in the range under consideration.

In \cite{BT} the following result was proved:

\begin{theorem} \label{th1} Let $1\le p<2$ and let $(X_n)$ be a determining sequence of random variables such that $\|X_n\|_p=1$ for all $n\in \mathbb{N}$,  $\{|X_n|^p, n\ge 1\}$ is uniformly integrable and $X_n\to 0$ weakly in $L^p$. Let $\mu$ be the limit random measure of $(X_n)$.
Then there exists a subsequence $(X_{n_k})$ equivalent to the unit vector basis of $\ell^2$ if and only if
\begin{equation}\label{BTcond}
\int_{\mathbb{R}}x^2 d\mu(x)\in L^{p/2}.
\end{equation}
\end{theorem}


\smallskip
The basic step in the proof of Theorem \ref{th1} was the following asymptotic result for lacunary series:

\begin{theorem} \label{th2} Under the assumptions of Theorem \ref{th1} there exists, for every $\ve>0$, an increasing sequence $(n_k)$ of positive integers such that
$$ (1-\ve)\rho_k(a_1, \ldots, a_k) \le \|\sum_{i=1}^k a_i X_{n_i}\|_p\le (1+\ve)\rho_k(a_1, \ldots, a_k)$$
for every $k\ge 1$ and $(a_1, \ldots, a_k)\in \mathbb{R}^k$, where
\begin{equation}\label{rho1}
\rho_n(a_1, \ldots, a_n)=\| \sum_{i=1}^n a_i Y_i\|_p.
\end{equation}
\end{theorem}



Clearly, $\rho_n(a_1, \ldots, a_n)$ is a symmetric function and thus Theorem \ref{th2} implies the results on almost symmetric sequences in $L_p$ spaces in \cite{ber,gura1,jo4}.

Note that for $p\ge 2$ the assumptions of Theorem \ref{th2} imply that the limit distribution $F$ of $(X_n)$ satisfies $\int_\mathbb{R} |x|^p dF(x)<\infty$, and thus
$$ {\mathbb E} \left( \int_\mathbb{R} |x|^p d\mu(x)\right)<\infty,$$
which implies (\ref{BTcond}) by the monotonicity of the $L^p$ norm. Actually, for $p\ge 2$ the conclusion of Theorem \ref{th2} is easily to prove and the statement was known
in the literature as early as in the 1930's.


The  purpose of the present paper is to extend Theorem \ref{th2} for general nonlinear functionals $f_k (a_1 X_{n_1}, \ldots, a_k X_{n_k})$ of lacunary sequences $(X_{n_k})$. We will namely prove the following


\begin{theorem}\label{th3} Let $(X_n)$ be a determining sequence of r.v.'s
with limit exchangeable sequence $(Y_n)$ with limit random measure $\mu$ having mean 0 a.s.\ and expectation satisfying
\begin{equation}\label{finitevar}
\E\left(\int_{\mathbb{R}} x^2 d\mu(x)\right)^{1/2}<\infty.
\end{equation}
Let $(S, d)$ denote the space of probability measures with mean $0$ and finite second moment equipped with the Wasserstein $W_2$ metric 
$$ d(\nu, \lambda)=\left( \int_0^1 (F_{\nu}^{-1} (x) - F_{\lambda}^{-1} (x))^2 dx \right)^{1/2}, $$
where $F_\nu$ and $F_\lambda$ denote the distribution functions belonging to $\nu$ and $\lambda$.
Let $f_k: \mathbb{R}^k\to\mathbb{R}$ ($k=1, 2, \ldots)$ be nonnegative measurable functions satisfying the condition
\begin{align} \label{a2}
& |\E f_{k+1} (a_1 x_1,  a_2x_2, \ldots,  a_{j-1}x_{j-1},  a_j t, a_{j+1}\xi_{j+1}^{(\nu)}, \ldots,  a_k\xi_k^{(\nu)}) \nonumber \\
&\phantom{999999999999}- \E f_{k+1} (a_1 x_1,  a_2x_2, \ldots, a_{j-1}x_{j-1},  a_j t', a_{j+1}\xi_{j+1}^{(\lambda)}, \ldots,  a_k\xi_k^{(\lambda)})|\nonumber \\
&\le a_j |t-t'| +\E f_k (a_1 Y_1, \ldots, a_k Y_k)\, d (\nu, \lambda)
\end{align}
for every $k\ge 1$, $1\le j\le k$, $\nu,\lambda \in S$, real numbers $t, t',  a_1,\ldots,a_k, x_1, \ldots, x_{j-1}$
and {\rm i.i.d.}\ sequences $(\xi_n^{(\nu)}),(\xi_n^{(\lambda)})$ with
respective distributions $\nu$ and $\lambda$. (The finiteness of the expectations in (\ref{a2}) is assumed here.)
Define, generalizing  (\ref{rho1}),
\begin{equation}\label{rhodef}
\rho_k (a_1, \ldots, a_k):=\E f_k (a_1 Y_1, \ldots, a_k Y_k),
\end{equation}
and assume that
\begin{equation} \label{lowerbound}
\rho_k (a_1, \ldots, a_k) \ge c|a_j|  \ \ 1\le j \le k
\end{equation}
for some constant $c>0$.
Assume finally that $\{|X_n|,\, n\ge 1\}$  is uniformly integrable and
$$ \E f(a_1 X_1, \ldots a_n X_n) <\infty  \ \ \text{for all} \ \ n\ge 1 \ \ \text{and} \ \ (a_1, \ldots, a_n)\in \mathbb{R}^n. $$
Then there exists, for every $\e>0$, a subsequence $(X_{n_k})$ satisfying
\begin{equation}\label{limit}  (1-\e) \rho_k(a_1, \ldots, a_k) \le  \E f_k (a_1 X_{n_1}, \ldots, a_k X_{n_k})
\le (1+\e) \rho_k (a_1, \ldots, a_k)
\end{equation}
for any $k\ge 1$ and any $(a_1,\ldots,a_k)\in {\mathbb R}^k$.
\end{theorem}


\medskip
By de Finetti's theorem, exchangeable sequences are conditionally i.i.d.\ with respect to their tail $\sigma$-algebra  and consequently they satisfy all limit theorems of i.i.d.\ random variables in a mixed form. By a heuristic principle {\it ('subsequence principle')}, formulated by Chatterji in \cite{cha},
the same holds for sufficiently thin subsequences of any sequence of r.v.'s satisfying norm boundedness conditions. Limit theorems for a sequence $(X_n)$ of random variables are asymptotic results for a sequence of functionals $f_k(X_1, X_2, \ldots)$ $(k=1, 2, \ldots)$ of the sequence $(X_n)$ and under suitable technical assumptions on the functionals $f_k$, Aldous \cite{ald} gave a rigorous proof of
the subsequence principle.
Note that relation (\ref{limit}) in Theorem \ref{th3} is also an asymptotic relation, except that it involves weighted functionals  $f_k (a_1 X_{n_1}, \ldots, a_k X_{n_k})$ and the limit relation is uniform in $(a_1, a_2, \ldots)$, a fact crucial for applications in Banach space theory.
The uniformity of the result is guaranteed by the equicontinuity relation (\ref{a2}).



In the case when  $f_k (x_1, \ldots, x_k)= |x_1+ \ldots +x_k|$,  the equicontinuity relation (\ref{a2}) was verified in \cite{BT} using the Marcinkiewicz-Zygmund inequality
\begin{equation} \label{mzz}
C \|\xi\|_1  \left( \sum_{i=1}^k a_i^2\right)^{1/2} \le \left\| \sum_{i=1}^k a_i \xi_i \right\|_p \le \| \xi\|_2 \left( \sum_{i=1}^k a_i^2\right)^{1/2}
\end{equation}
for $p=1$; here $(\xi_n)$ is an i.i.d.\ sequence with mean zero, and finite second moment. This yields Theorem 1.2 for $p=1$. For $f_k (x_1, \ldots, x_k)= |x_1+ \ldots +x_k|^p$, $1\le p< 2$, an estimate for the left hand side of (\ref{a2}) is given in \cite{BT}, page 2068, but this estimate is of a different form as (\ref{a2}). Thus, purely  formally,
Theorem \ref{th2} is not a special case of Theorem \ref{th3} for $1<p<2$. Still, the proof of Theorem \ref{th3} works in this case, since by the estimate for $|g^{{\bf a}, \ell}(t, \nu) - g^{{\bf a}, \ell} (t'\ \nu')|$ in \cite{BT}, p.\ 2061 the class of functions in (\ref{29c})  satisfies assumptions (a), (b) of the subsequent Lemma \ref{rrlemma}, and this is all we need for the argument.

For nonlinear functionals $f_k$ the verification of (\ref{a2}) is generally technically complicated, but intuitively it is clear that (\ref{a2}) holds for a very large class of functionals $f_k$. Note that (\ref{a2}) holds trivially for $f_k(x_1, \ldots, x_k)=\max_{1\le j\le k} |x_j|.$

A further interesting consequence of Theorem 1.3 is the following


\begin{theorem} \label{th4}
Let $\psi$ be a Young function with $x\ll \psi(x) \ll x^p$,  $p\in[1,2)$ and let $(X_n)$ be a determining sequence such that $\|X_n\|_{L^\psi}=1$ for all $n\in \mathbb{N}$, , $\{\psi(|X_n|), n\ge 1\}$ is uniformly integrable and $X_n\to 0$ weakly in $L^\psi$. Let $\mu$ be the limit random measure of $(X_n)$.
Then there exists a subsequence $(X_{n_k})$ equivalent to the unit vector basis of $\ell^2$ if and only if
$$ \int_{\mathbb{R}} x^2 d\mu(x) \in L^{\sqrt{\psi}}.$$
\end{theorem}

\bigskip
We also note that under $\psi(x) \gg x^p$, $p>2$, Kadec-Pe{\l}czynski-type results also hold under the $L^\psi$ norm, see e.g \cite{JA,KL}.

By a result of Koml\'os \cite{KO}, every sequence $(X_n)$ of random variables with bounded $L_2$ norms satisfying $X_n\to 0$  weakly in $L_2$ has a subsequence $(X_{n_k})$ which is an unconditional convergence sequence, i.e. $\sum_{k=1}^\infty c_k X_{n_k}$ converges a.s.\ after any permutation of its terms provided $\sum_{k=1}^\infty c_k^2<\infty$. This was an unsolved problem originating from the 1930's, see Ulyanov \cite{UL}, p.\ 54. A simple proof of Koml\'os' result was given by Aldous \cite{ald}, using an equicontinuity argument. A similar method  was used in \cite{BT2} to prove uniform limit theorems for normed lacunary sums $A_N^{-1}\sum_{k=1}^N a_k X_{n_k}$ with limit distribution having Fourier transform $\exp (-c|t|^\alpha)$, $0<\alpha<2$, $c>0$.

\section{ Proof of Theorem \ref{th3}.}
We follow the argument in \cite{BT}.
We will need the following

\begin{lemma}  \label{rrlemma} (Ranga Rao \cite{rara}).
Let $(S, d)$ be a separable metric space and $\nu, \nu_n$ $(n=1, 2, \ldots)$ probability measures on the Borel sets of $(S, d)$ such that $\nu_n\overset{\mathcal D}{\longrightarrow} \nu$. Let $\mathcal G$ be a cass of real valued functions on $(S, d)$ such that

\smallskip(a) $\mathcal{G}$ is locally equicontinuous.

\smallskip (b) There exists a continuous function $g\ge 0$ on $S$ such that $|f(x)|\le g(x)$ for all $f \in \mathcal G$ and $x\in S$ and
$$ \int_S g(x)\nu_n (dx)\longrightarrow \int_S g(x)\nu(dx)\  (<\infty) \ \ \text{as} \ \ n\to\infty.$$
Then
$$ \int_S f(x)\nu_n (dx) \longrightarrow \int_S f(x)\nu(dx) \ \ \text{as} \ \ n\to\infty$$
uniformly in $f\in \mathcal G$.
\end{lemma}

Let now $(\Omega,{\mathcal F}, P)$ be the probability space of the $X_n$'s and ${\bf
X} = (X_1,X_2,\ldots)$.
Let $(Y_n)$ be a sequence of
r.v.'s on $(\Omega ,{\mathcal F}, P)$ such that, given {\bf X} and
$\mu$, the r.v.'s $Y_1,Y_2,\ldots\ $ are conditionally i.i.d.\
with distribution $\mu$, {\it i.e.},

\begin{equation}\label{18c}
P(Y_1\in A_1,\ldots,Y_k\in A_k\vert {\bf X},\mu) = \prod_ {i=1}^k
P(Y_i\in A_i\vert {\bf X},\mu) \quad \textup{a.s.}
\end{equation}
\begin{equation}\label{19c}
P(Y_j\in A\vert {\bf X},\mu )= \mu (A) \quad \textup{a.s.}
\end{equation}
for any $j,k$ and Borel sets $A,A_1,\ldots,A_k$ on the real line.
Such a sequence $(Y_n)$ always exists after a suitable enlargement
of the probability space, see e.g.\ \cite{ald}, p.\ 72.
Obviously, $(Y_n)$ is an exchangeable sequence and (\ref{18c}) and (\ref{19c}) imply
\begin{equation}\label{5c}
P(Y_{i_1}\in I_1,\ldots,Y_{i_k}\in I_k) = \E\left( \mu (I_1)\ldots \mu
(I_k)\right)
\end{equation}
for every  $k\ge 1$, $i_1<\cdots <i_k$ and left closed intervals
$I_1,\ldots,I_k \subset \R^1$.

We recall also the following lemma from \cite{BT}:
\begin{lemma}\label{alemma}
For every $\sigma({\mathbf X})$-measurable random variable $Z$ and any $j\ge 1$ we have
$$ (X_n, Z) \overset{\mathcal D}{\longrightarrow} (Y_j, Z) \ \ \text{as} \ n\to\infty.$$
\end{lemma}

\bigskip
Fix now $0<\varepsilon\le 1/2$.  We will construct a sequence $n_1<n_2<\cdots \ $ of integers such that
\begin{equation} \label{24c}
(1-\varepsilon) \rho_k (a_1,\ldots,a_k) \le \E f_k (a_1
X_{n_1}, \ldots a_k X_{n_k}) \le (1+\varepsilon)\rho_k (a_1,\ldots,a_k)
\end{equation}
for every $k\ge 1$ and $(a_1,\ldots,a_k)\in \R^k$. where
$\rho_k$ is defined by (\ref{rhodef}).
To construct $n_1$ we set
\begin{align*}
Q({\bf a},n,\ell) = f_\ell (a_1 X_n, a_2Y_2, \ldots, a_\ell Y_\ell), \qquad
R({\bf a},\ell) = f_\ell (a_1Y_1, a_2Y_2, \ldots, a_\ell Y_\ell)
\end{align*}
for every $n\ge 1$, $\ell \ge 2$ and ${\bf a} =
(a_1,\ldots,a_\ell) \in \R^\ell$.  We show that
\begin{equation}\label{25c}
\E \left\{ \frac{Q({\bf a},n,\ell)}{\rho_\ell ({\bf a})} \right\}
\longrightarrow   \E \left\{ \frac{R({\bf a},\ell)}{\rho_\ell ({\bf
a})} \right\} \ \hbox{ as }\ n\to \infty \quad \hbox{uniformly
in }\ {\bf a}\ne {\bf 0},\ell.
\end{equation}
Note that the right hand side of (\ref{25c}) equals 1 and for ${\bf a}\ne {\bf 0}$ we have $\varrho_\ell ({\bf a})\ne {\bf 0}$ by relation  (\ref{lowerbound}).   To this end we recall
that, given {\bf X} and $\mu$, the r.v.'s $Y_1,Y_2,\ldots \ $ are
conditionally i.i.d.\ with common conditional distribution $\mu$
and thus, given ${\bf X},\mu$ and $Y_1$, the r.v.'s
$Y_2,Y_3,\ldots \ $ are conditionally i.i.d.\ with distribution
$\mu$.  Thus
\begin{equation*}
\E \bigl( Q({\bf a},n,\ell)\vert {\bf X},\mu\bigr) =g^{{\bf a},\ell} (X_n,\mu)
\end{equation*}
and
\begin{equation}\label{27c}
\E \bigl( R({\bf a},\ell)\vert {\bf X},\mu,Y_1\bigr) = g^{{\bf
a},\ell} (Y_1,\mu)
\end{equation}
where
\begin{equation*}
g^{{\bf a},\ell} (t,\nu) = \E f_\ell (a_1 t, a_2\xi_2^ {(\nu)}, \ldots, a_\ell \xi_\ell^{(\nu)})
\qquad (t\in \R^1\ ,\ \nu \in S)
\end{equation*}
and $(\xi_n^{(\nu)})$ is an i.i.d.\ sequence with distribution
$\nu$. Integrating
the previous relations, we get
\begin{equation}\label{26c}
\E \bigl( Q({\bf a}, n,\ell)\bigr) = \E g^{{\bf a},\ell} (X_n,\mu)
\end{equation}
\begin{equation}\label{27d}
\E \bigl( R({\bf a},\ell)\bigr) = \E g^{{\bf a},\ell} (Y_1,\mu)
\end{equation}
and thus (\ref{25c}) is equivalent to
\begin{equation}\label{28c}
\E \frac{g^{{\bf a},\ell} (X_n,\mu)}{\rho_\ell ({\bf a})}
\longrightarrow \E \frac{g^{{\bf a},\ell} (Y_1,\mu)}{\rho_\ell ({\bf
a})} \ \hbox{ as }\ n\to \infty\ ,\ \hbox{ uniformly in }  \ \ell, {\bf a}\ne {\bf 0}.
\end{equation}
We shall derive (\ref{28c}) from Lemma \ref{rrlemma}.
It is easy to see that the metric $d$ in Theorem \ref{th3} is a separable metric on
$$ S=\left\{ \nu \in \mathcal{M} : \int_{\mathbb{R}} x d\nu (x)=0, \ \int_{\mathbb{R}} x^2 d\nu (x)<\infty \right\}$$
and it generates the same Borel $\sigma$-field on $S$ as the Prokhorov metric $\pi$. But then the limit random measure $\mu$, which is a random variable
taking values in $(S, \pi)$,
can also be regarded a as a random variable taking values in $(S, d)$.
Also, $\mu$ is clearly $\sigma(X)$ measurable and thus
$(X_n, \mu)\overset{\mathcal D}{\rightarrow} (Y_1, \mu)$ by Lemma \ref{alemma}. Hence (\ref{28c}) will follow from Lemma \ref{rrlemma} (note the equivalence of
(\ref{25c}) and (\ref{28c})) if we show that the class of functions
\begin{equation}\label{29c}
\left\{ \frac{g^{{\bf a},\ell}(t,\nu)}{\rho_\ell ({\bf a}) }
\right\}
\end{equation}
defined on the product metric space $(\R^1\times S\ ,\
\lambda^1\times d)$ ($\lambda^1$ denotes the ordinary distance on
$\R^1$) satisfies conditions (a),(b) of Lemma \ref{rrlemma}. The validity of (a) follows immediately from relations (\ref{a2}) and (\ref{lowerbound}).
On the other hand, using (\ref{a2}), (\ref{lowerbound}) we get for any $\nu \in S$, $t\in \mathbb{R}^1$ and
${\bf a} = (a_1,\ldots a_\ell)\in \mathbb{R}^\ell$, that


\begin{align*}
&|\E f_\ell (a_1t, a_2\xi_2^{(\nu)}, \ldots,  a_\ell\xi_\ell^{(\nu)})| \\
&\le |\E f_\ell (0, a_2\xi_2^{(\nu_0)}, \ldots,  a_\ell\xi_\ell^{(\nu_0)})|  + |a_1 t| + |\E f_\ell(a_1Y_1, \ldots, a_\ell Y_\ell)| d(\nu, \nu_0)\\
& = |a_1 t|+|\E f_\ell(a_1Y_1, \ldots, a_\ell Y_\ell)| d(\nu, \nu_0)\\
&  \le \rho_\ell ({\bf{a}})|t| + \rho_\ell ({\bf{a}})d(\nu, \nu_0),
\end{align*}
where $\nu_0$ denotes the probability measure concentrated at 0 and thus the $\xi_j^{(\nu_0)}$ are identically 0. Hence to verify assumption (b) of Lemma \ref{rrlemma} we need to check that if $(X_n, \nu)\overset{\mathcal D}{\rightarrow} (Y_1, \nu)$, then
$$ \E\left( |X_n|+ d(\nu, \nu_0)\right)\longrightarrow \E\left(|Y_1|+ d(\nu, \nu_0)\right).$$
This follows, however, from the uniform integrability of the $X_n$ and the fact that $d(\nu, \nu_0)=\E (\int_{\mathbb{R}}x^2 d\nu(x))^{1/2}<\infty$ by (\ref{finitevar}).
We thus proved relation (\ref{28c}) and thus also
(\ref{25c}) whence it follows (note again that the right side of
(\ref{25c}) equals 1) that
\begin{align}\label{34c}
&\rho_\ell ({\bf a})^{-1} \E f_\ell (a_1X_n, a_2Y_2, \ldots, a_\ell Y_\ell)
\cr
& \to \rho_\ell ({\bf a})^{-1} \E f_\ell (a_1Y_1, a_2Y_2, \ldots, a_\ell
Y_\ell) \ \hbox{ as }\ n\to \infty
\end{align}
uniformly in $\ell, {\bf a}$.  Hence we can choose $n_1$ so large
that
$$
\E f_\ell (a_1X_{n_1}, a_2Y_2,  \ldots, a_\ell Y_\ell)
-\E f_\ell (a_1Y_1, a_2Y_2, \ldots, a_\ell Y_\ell)
\big\vert \le \frac{\varepsilon}{2} \rho_\ell (a_1,\ldots,a_\ell)$$
 for every ${\bf a}\ne {\bf 0}$ and $\ell\ge 1$. The last relation clearly holds for $\bf a={\bf 0}$, and thus the first induction step is completed.

Assume now that $n_1,\ldots ,n_{k-1}$ have already been chosen.
Exactly in the same way as we proved (\ref{34c}), just putting $a_kX_n $ at the $k$-th location,
it follows that for $\ell >k$
\begin{align*}&\rho_\ell ({\bf a})^{-1} \E f_\ell (a_1X_{n_1}, \ldots, a_{k-1}X_{n_{k-1}},
a_k X_n, a_{k+1}Y_{k+1}, \ldots, a_\ell Y_\ell ) \cr
&\longrightarrow \rho_\ell ({\bf a})^{-1} \E f_\ell ( a_1X_{n_1}, \ldots,
a_{k-1} X_{n_{k-1}},  a_kY_k, \ldots,  a_\ell Y_\ell) \ \hbox{
as }\ n\to \infty\cr
\end{align*}
uniformly in {\bf a} and $\ell$.  Hence we can choose $n_k$ so
large that $n_k>n_{k-1}$ and
\begin{align*}&\big\vert \ \E f_\ell (a_1X_{n_1}, \ldots,  a_{k-1}X_{n_{k-1}},
a_kX_{n_k},  a_{k+1}Y_{k+1}, \ldots,  a_\ell Y_\ell) \cr
&\qquad - \E f_\ell (a_1X_{n_1}, \ldots, a_{k-1}X_{n_{k-1}},  a_kY_k,
\ldots, a_\ell Y_\ell) \big\vert \le \frac{\varepsilon}{2^k}
\rho_\ell (a_1,\ldots ,a_\ell) \cr
\end{align*}
for every $(a_1,\ldots,a_\ell)\in \R^\ell$ and $\ell >k$.  This
completes the $k$-th induction step; the so constructed sequence
$(n_k)_{k \geq 1}$ clearly satisfies
$$\big\vert \ \E f_\ell (a_1X_{n_1}, \ldots,  a_\ell X_{n_\ell})
- \E f_\ell (a_1Y_1, \ldots, a_\ell Y_\ell) \vert \le \varepsilon
\rho_\ell (a_1,\ldots ,a_\ell)$$ for every $\ell \ge 1$  and
$(a_1,\ldots, a_\ell)\in \R^\ell$.
The last relation is equivalent to (\ref{24c}) and thus Theorem \ref{th3} is proved.


\begin{proof}[Proof of Theorem \ref{th4}]

Since $\psi$ is convex and non-decreasing, the right derivative $\psi_+'(t)$ exists for every $t>0$, and for
$0<x<y$ we have
\[
\frac{\psi(y)-\psi(x)}{y-x}\le \psi_+'(y).
\]
Moreover,
\[
\psi_+'(t)\le \frac{\psi(2t)-\psi(t)}{t}\le \frac{\psi(2t)}{t}\ll t^{p-1},
\qquad t>0.
\]
Hence, for $0<x<y$,
\[
\psi(y)-\psi(x)\ll (y-x)y^{p-1}.
\]
Since
\[
y^{p-1}\le x^{p-1}+y^{p-1},
\]
it follows that
\[
\psi(y)-\psi(x)\ll |x-y|\bigl(x^{p-1}+y^{p-1}\bigr).
\]
Interchanging $x$ and $y$ if necessary, we obtain
\begin{equation}\label{2n}
|\psi(x)-\psi(y)| \ll |x-y|\bigl(x^{p-1}+y^{p-1}\bigr),
\qquad x,y>0.
\end{equation}
Letting
\begin{equation}\label{newdef}
h^{{\bf a}, l} (t, \nu)= \E\psi\bigl(|a_1t +\sum_{i=2}^\ell a_i \xi_i^{(\nu)}\bigr|\bigr),
\end{equation}
relation (\ref{2n}) implies that the estimates for $|g^{{\bf a}, l} (t, \nu)|$ and $|g^{{\bf a}, l} (t, \nu)- g^{{\bf a}, l} (t', \nu')|$ obtained in 
\cite{BT}, page 2060 remain valid for $|h^{{\bf a}, l} (t, \nu)|$ and $|h^{{\bf a}, l} (t, \nu)- h^{{\bf a}, l} (t', \nu')|$ 
and consequently the class of functions
\begin{equation*}
\left\{ \frac{h^{{\bf a},\ell}(t,\nu)}{  \|{\bf a}\|_2^p }
\right\}
\end{equation*}
satisfies conditions (a) and (b) of Lemma 2.1. Thus the induction procedure on pp.\ 2060-2062 of \cite{BT} remains valid and leads to a subsequence $(X_{n_k})$  such that for any $\varepsilon > 0$ we have
\[
\mathbb{E}\psi\left(\left|\sum_{i=1}^k a_i X_{n_i}\right|\right)
\overset{1+\varepsilon}{\approx}
\mathbb{E}\psi\left(\left|\sum_{i=1}^k a_i Y_i\right|\right),
\]
where the symbol $\overset{1+\varepsilon}{\approx}$ means that the ratio of the two sides is
between $1-\varepsilon$ and $1+\varepsilon$.
Replacing $\psi(x)$ by $\psi(x/C)$ for any fixed $C > 0$ means only a scale change and thus the subsequence $(X_{n_k})$ provided by the induction procedure will satisfy 
also 
\[
\mathbb{E}\psi\left(\frac{\left|\sum_{i=1}^k a_i X_{n_i}\right|}{C}\right)
\overset{1+\varepsilon}{\approx}
\mathbb{E}\psi\left(\frac{\left|\sum_{i=1}^k a_i Y_i\right|}{C}\right)
\]
for all $C>0$ simultaneously. Let now
\[
f = \sum_{i=1}^k a_i X_{n_i},
\qquad
g = \sum_{i=1}^k a_i Y_i.
\]
Since \(\psi\) is continuous, if \(C^*\) denotes the \(L^\psi\)-norm of \(g\), then
$\mathbb{E}\,\psi\!\left(|g|/C^*\right)=1,$
and thus
\begin{equation}\label{5n}
1-\varepsilon \le \mathbb{E}\psi(|f|/C^*) \le 1+\varepsilon.
\end{equation}
Let $C^{**}$ denote the $L^\psi$ norm of $f$, then
\begin{equation}\label{6n}
\mathbb{E}\psi(|f|/C^{**}) = 1.
\end{equation}
Let
\[
G(c):=\mathbb{E}\psi(c|f|), \qquad
a:=\frac{1}{C^{**}}, \qquad b:=\frac{1}{C^*}.
\]
Then
\[
G(a)=1,
\qquad
1-\varepsilon \le G(b)\le 1+\varepsilon.
\]
Since \(G\) is convex, which follows from the convexity of $\psi$ and since \(G(0)=0\), the function
\[
c \mapsto \frac{G(c)}{c}
\]
is nondecreasing on \([0,\infty)\). Indeed, for \(0<x<y\), and by convexity of \(G\),
\[
G(x)
=
G\!\left(\frac{x}{y}\,y+\left(1-\frac{x}{y}\right)0\right)
\le
\frac{x}{y}G(y)+\left(1-\frac{x}{y}\right)G(0)
=
\frac{x}{y}G(y).
\]
Dividing by \(x>0\), we obtain
\[
\frac{G(x)}{x}\le \frac{G(y)}{y},
\]
which proves the claim.

If \(b\le a\), then
\[
\frac{G(b)}{b}\le \frac{G(a)}{a}=\frac{1}{a},
\]
so
\[
1-\varepsilon \le G(b)\le \frac{b}{a}=\frac{C^{**}}{C^*}\le 1.
\]
If \(b\ge a\), then
\[
\frac{G(b)}{b}\ge \frac{G(a)}{a}=\frac{1}{a},
\]
so
\[
1\le \frac{C^{**}}{C^*}=\frac{b}{a}\le G(b)\le 1+\varepsilon.
\]
Therefore,
\[
1-\varepsilon \le \frac{C^{**}}{C^*}\le 1+\varepsilon,
\]
i.e.\ the ratio of the $L^\psi$ norm of $f$ and $g$ is between positive constant bounds.
In other words,
\begin{equation} \label{xni}
K_1\Bigl\|\sum_{i=1}^k a_i Y_i\Bigr\|_{L^\psi}
\le
\Bigl\|\sum_{i=1}^k a_i X_{n_i}\Bigr\|_{L^\psi}
\le
K_2\Bigl\|\sum_{i=1}^k a_i Y_i\Bigr\|_{L^\psi}
\end{equation}
for all $k\ge 1$ and all $(a_1, \ldots, a_k)\in \mathbb{R}^k$ with some positive constants $K_1, K_2$.
We now claim 
\begin{equation}\label{mzphi}
C \|Y\|_{1}
\left(\sum_{i=1}^k a_i^2\right)^{1/2}
\le \left\|\sum_{i=1}^k a_i Y_i\right\|_{L^\psi}
\le
\|Y\|_{2}
\left(\sum_{i=1}^k a_i^2\right)^{1/2}
\end{equation}
for some constant $C>0$. In the case when the $Y_j$ are i.i.d.\ random variables with mean 0 and finite variance, the upper bound in (\ref{mzphi}) follows from the fact that by the relation $x \ll \psi(x) \ll x^2$ and the embedding relation (\ref{embedding}), the $L^2$ norm dominates the $L^\psi$ norm. The lower bound, on the other hand, follows from the fact that the $L^\psi$ norm dominates the $L^1$ norm, thereby reducing the problem to the already proven inequality (\ref{mzz}) for $p=1$. Since the limit exchangeable sequence $(Y_j)$ is conditionally i.i.d.\ with respect to its tail $\sigma$-algebra with conditional mean zero and conditional finite variance, the validity of (\ref{mzphi}) in the exchangeable case follows from the i.i.d.\ case by integration. Combining (\ref{xni}) and (\ref{mzphi} we get  
$$\left(\sum_{i=1}^k a_i^2\right)^{1/2}  \ll \left\|\sum_{i=1}^k a_i X_{n_i}\right\|_{L^\psi} \ll \left(\sum_{i=1}^k a_i^2\right)^{1/2} $$
and thus the subspace spanned by $(X_{n_k})$ in $L^\psi$ is equivalent to $\ell^2$.


The necessity part of Theorem 1.1 in \cite{BT} was deduced from the relation
$$ \left\| \frac{1}{\sqrt{N}} \sum_{k=1}^N X_{n_k} \right\|_p = O(1) $$
on p.\ 2062. Since by the above explanation the last relation remains valid under the Orlicz norm, the proof of the necessity part of Theorem \ref{th4} follows again the same way as in \cite{BT}.
\end{proof}

\begin{rem}
   We remark that assumption \eqref{a2} implies that Theorem \ref{th3} recovers Theorem \ref{th2} only in the case \(p=1\). However, Theorem \ref{th4} fully covers Theorems \ref{th1} and \ref{th2}.
\end{rem}



\textbf{Acknowledgments}
The authors would like to express their sincere gratitude to the referee for the careful reading of the manuscript and for many insightful comments and valuable suggestions, which led to a substantial improvement of the paper. ES expresses gratitude to Christoph Aistleitner and Andrei Shubin for helpful discussions.

\bibliographystyle{siam}

\end{document}